\documentclass[12pt]{amsart}
\usepackage{hyperref}
\usepackage{xcolor}
\usepackage{graphicx}
\usepackage[utf8x]{inputenc}
\usepackage[T1,T2A]{fontenc}
\usepackage[russian,english]{babel}
\usepackage{babelbib}
\usepackage[margin=3cm]{geometry}
\usepackage{amsmath,amssymb,amsthm,amsfonts}

\newtheorem{theorem}{Theorem}
\newtheorem{proposition}{Proposition}

\newtheorem{lemma}{Lemma}

\newcommand{\gf}{\mathfrak{g}}

\newcommand{\hf}{\mathfrak{h}}

\newcommand{\pa}{\partial}

\newcommand{\RR}{\mathbb{R}}

\title[Tensor products]{On multiplicities of irreducibles in large tensor product of representations of simple Lie algebras.}

\author{Olga Postnova}
\address{O.P.:Laboratory of Mathematical Problems of Physics,
St. Petersburg Department of Steklov Mathematical Institute,191023, Fontanka 27, St. Petersburg, Russia}
\email{postnova.olga@gmail.com}

\author{Nicolai Reshetikhin}
\address{N.R.: Department of Mathematics, University of California, Berkeley,
CA 94720, USA \& Physics Department, St. Petersburg University, Russia \&KdV Institute for Mathematics, University of Amsterdam,
Science Park 904, 1098 XH Amsterdam, The Netherlands.}
\email{reshetik@math.berkeley.edu}

\begin{document}

\maketitle
\begin{abstract}

In this paper we study the asymptotic of multiplicities 
of irreducible representations in large tensor products of finite dimensional representations 
of simple Lie algebras and their statistics with respect to Plancherel and character probability measures.  We derive the 
asymptotic distribution of irreducible components for the Plancherel measure, generalizing 
results of Biane and Tate and Zelditch. We also derive the asymptotic of the character measure for generic parameters
and an intermediate  scaling in the vicinity of the Plancherel measure.  
It is interesting that the asymptotic measure 
is universal and after suitable renormalization does not depend on which representations were multiplied but depends significantly  on the degeneracy of the parameter in the character distribution.

\end{abstract}

\section{Introduction}
\label{sec:introduction}

\subsection{} The study of the statistics of irreducible components in ``large"
natural representations goes back to works
\cite{LS}\cite{VK1}\cite{VK2}. An example of such large representation is the left regular representation of the symmetric group $S_N$ for large $N$. The Plancherel measure provides a natural probability measure.
The statistics of irreducible components in the left regular representation of $S_N$ with respect to the
Plancherel measure was exactly the focus of  \cite{LS}\cite{VK1}\cite{VK2},\cite{BOO}, see also \cite{BG}\cite{M} and references therein.

The Schur-Weyl duality identifies multiplicities of irreducible components of the $N$-th tensor power of the vector representation of 
$sl_{n+1}$ with dimensions of irreducible representations of $S_N$. The statistics of irreducible  $sl_{n+1}$-components
in this tensor products with respect to the Plancherel measure\footnote{If $W$ is a finite dimensional representation of a simple Lie algebra and $W\simeq\oplus_iV_i^{\oplus m_i}$ is its decomposition into irreducibles, then $p_i=\frac{\dim V_im_i}{\dim W}$ is a probability measure on the set of irreducible components of $W$. We call it the Plancherel measure by the analogy with the Plancherel measure on the left regular representation of a finite group, or compact Lie group.} (\ref{Plan})
for large $N$ was studied by Kerov in \cite{K1}.  Kerov \cite{K1}  discovered that as $N\to \infty$ the discrete
measure (\ref{Plan}) on the space of highest weights $\mathbb{R}^{n}$ of $sl_{n+1}$ converges weakly to the measure on the main Weyl chamber given by the radial part of the natural invariant measure on $sl_{n+1}^*$. Kerov's 
proof was based upon the hook formula
for dimensions of irreducible representations of $S_N$.

Among more recent results involving similar computations are papers \cite{BR1},\cite{BR2} where similar asymptotics
were computed for Lie superalgebras $gl(n|m)$, the paper \cite{NP}, where the asymptotic of multiplicities of irreducible 
$so(2n+1)$-modules in the $N$-th tensor powers of the spinor representation  was studied in the limit $N\to\infty$
and \cite{BG}.
In \cite{B} the asymptotic of multiplicities of irreducible subrepresentations was computed for $V^{\otimes N}$
when $N\to \infty$. These results were extended and strengthened in \cite{TZ}, see also \cite{ST}. 
Problems of this type are known as {\it asymptotic representation theory}.
A survey of this direction in representation theory can be found in the books \cite{K2}, \cite{V}. 

In this paper we derive the asymptotic formula for multiplicities of irreducible representations in tensor powers of finite dimensional representations of simple Lie algebras when the number of factors tends to infinity. The formula
generalizes slightly the one obtained in \cite{TZ},\cite{B}. Here we derive it using a different method, somewhat heuristically. The complete proof is a straightforward extension of results from \cite{TZ}. The large deviation rate function was
also computed in \cite{Du} and the convergence of character measures was studied in \cite{KW}.

\subsection{} Let $\gf$ be a simple Lie algebra, $V_i,\;\;i=1\dots m$ be its finite dimensional representations and $N_k\geq 0$ be integers.  Any finite dimensional representation of a simple Lie algebra is completely reducible and therefore:

\begin{equation} 
\label{tp}
\bigotimes_{k=1}^{m} V_k^{\otimes N_k} \cong \bigoplus_{\lambda}W_{\lambda}(\{V_k\}, \{N_k\}) \otimes V_{\lambda},
\end{equation}
where the sum is taken over irreducible components of the tensor product, $V_{\lambda}$ is the irreducible finite dimensional $\gf$-module with the highest weight $\lambda$ and
$W_{\lambda}(\{V_k\}, \{N_k\})$ is the "space of multiplicities":
\[
W_{\lambda}(\{V_k\}, \{N_k\})\simeq \mathrm{Hom}_{\gf}( \bigotimes_{k=1}^{m}V_k^{\otimes N_k}, V_\lambda).
    \]
Its dimension $m_{\lambda}(\{V_k\}, \{N_k\})$ is the multiplicity of
$V_{\lambda}$ in the tensor product and we have isomorphism of vector spaces.  

Starting from here we assume $V_k=V_{\nu_k}$, otherwise $V_k\simeq\bigoplus V_{\nu}^{\oplus m_{\nu,k}}$.
We will use notations $\chi_\lambda(\gf)=tr_{V_\lambda(\gf)}$ for characters of irreducible representations.

Choose a Borel subalgebra $\bf\in \gf$ and let $\Delta_+$ be corresponding positive roots, $\hf$ be
the corresponding Cartan subalgebra and $\alpha_1,\dots, \alpha_r$ be enumerated fundamental roots.
Here $r=rank(\gf)=dim(\hf)$ is the rank of the Lie algebra $\gf$.
Let $\gf_{\mathbb{R}}$ be the split real form of $\gf$ and $\hf_{\RR}$ be its Cartan subalgebra. When $g=e^t,\;\;t \in \subset\hf_{\mathbb{R}}\subset\gf_{\mathbb{R}}$ the characters $\chi_{\nu_k}(e^t)$ and $\chi_\lambda(e^t)$ are positive and
as a consequence of the tensor product decomposition we have the identity
    \[
        \prod_k \chi_{\nu_k}(e^t)^{N_k}=\sum_{\lambda} m_{\lambda}(\{V_k\}, \{N_k\})\chi_\lambda(e^t) .
    \]

Therefore
\begin{equation}
  \label{stat}
   p_\lambda^{(N)}(t) =  \frac{m_{\lambda}(\{V_k\}, \{N_k\})\chi_\lambda(e^t)}{\prod_k \chi_{\nu_k}(e^t)^{N_k}}
\end{equation}
is a natural probability measure on irreducible components of tensor product: $p_\lambda^{(N)}(t)\geq 0$ and $\sum_\lambda p_\lambda^{(N)}(t)=1$. We will call it {\it the character measure}.

We  will also assume\footnote{This is a convenience assumption, the irreducible characters restricted
to the Cartan subgroup are 
invariant with respect to the action of the Weyl group, so we can always choose a representative
of the orbit of $W$ through $t$ which is positive.} that components of $t$ in the basis of simple roots
are nonnegative, i.e. $t\geq 0$.

In this paper we study the asymptotical behavior of multiplicities $m_\lambda$  in the limit when $N_k\to \infty$ and $\lambda \to\infty$ such that $N_k=\tau_k/\epsilon$ and $\lambda=\xi/\epsilon,\;\;\epsilon\to 0$, where $\tau_k\in \RR_{\geq 0}$ and $\xi\in \hf^*_{\geq 0}$. As a consequence we will describe
the large deviation type asymptotic of the probability measure $\{p_\lambda\}$.

\subsection{} To state main results of the paper we need a few definitions.
First, define 
\[
    f(\tau,t)=\sum_k\tau_k\ln(\chi_{\nu_k}(e^t)) .
\]
It is a strictly convex function of $t$. Let $S(\tau, \xi)$ be the  Legendre transform of $f(\tau, y)$
\begin{equation}\label{Sf}
    S(\tau,\xi)=\min_y\left(f(\tau,y)-(y,\xi)\right)=
    f(\tau,x)-(x,\xi),
\end{equation}
where $(y,\xi)$ is the Killing form. In the basis of simple roots $\xi=\sum_a{\xi_a\alpha_a}, y=\sum_ay_a\alpha_a$ and $(y,\xi)=\sum_{ab}y_aB_{ab}\xi_b$, where $B_{ab}=d_aC_{ab}=(\alpha_a,\alpha_b)$ is the symmetrized Cartan matrix and $x$ is the critical point where the minimum is achieved. The critical point $x$ is the unique solution to the equation:
\begin{equation}
\label{tt}
\frac{\partial}{\partial x_a}f(\tau,x)=\sum_b B_{ab}\xi_b.
\end{equation}
Define the matrix $K$ as
\[
K_{ac}=-\frac{\pa^2 S(\tau, \xi)}{\pa \xi_a\pa \xi_b},
\]
we will see that 
\[
K_{ac}=   \sum_{b,d}B_{ad}(D^{-1})_{db}B_{bc},
\]
where
\begin{equation*}
    D_{ab}=\frac{\partial^{2}}{\partial y_a \partial y_b} f(\tau,y)|_{y=x},
\end{equation*}
where $x$ is as above. 

Recall that $y\in \hf$ is called {\it regular} if
the stabilizer of $y$ in the Weyl group $W$ acting on $\hf$ is trivial. The following theorem 
extends the result of \cite{B},\cite{TZ} where it was proven for $m=1$. 

\begin{theorem} \label{con-1}If $\xi=\epsilon \lambda$ remain regular as $\epsilon\to 0$ 
the asymptotic of the multiplicity of $V_\lambda$ in (\ref{tp}) has the following form
\[
    m_{\lambda}(\{V_k\}, \{N_k\})=
    \epsilon^{\frac{r}{2}}\frac{\sqrt{detK}}{(2\pi)^{\frac{r}{2}} }\Delta(x)e^{-(\rho,x)}
e^{\frac{1}{\epsilon}S(\tau,\xi)}(1+O(\epsilon))
\]
Here $x\in \hf$ is the Legendre image of $\xi\in \hf^*$, the functions $S$ and the matrix $K$ are
as above and $\Delta(x)$ is the denominator in the Weyl formula for characters:
\[
\Delta(x)=\prod_{\alpha\in \Delta_+} (e^{\frac{(x,\alpha)}{2}}-e^{-\frac{(x,\alpha)}{2}})
\]
\end{theorem}

Note that the asymptotic is different when $\xi$ is not regular. In this paper we give a heuristic  argument why the
theorem holds. A rigorous proof, based on the application of the steepest descent method to the integral of characters
can be found in \cite{TZ} for $m=1$ and for $m>1$ it is completely parallel. 

The next theorem is the description of the statistics of irreducible components
with respect to the character measure (\ref{stat}).

Assume $\xi=\epsilon \lambda$ remain regular as $\epsilon\to 0$ and $t$ is regular.
Taking into account the asymptotic of the multiplicity and the asymptotic of the charachter $\ch_\lambda(e^t)$
in this limit, we obtain the following
asymptotic of the probability $p_\lambda^{(N)}(t)$ as $\epsilon\to 0$  is:
\begin{equation}\label{A}
p_\lambda^{(N)}(t)=\epsilon^{\frac{r}{2}}\frac{\sqrt{detK}}{(2\pi)^{\frac{r}{2}}}\frac{\Delta(x)}{\Delta(t)}e^{-(\rho,x-t)}
e^{\frac{1}{\epsilon}\widetilde{S}(\tau,\xi)}(1+O(\epsilon))
\end{equation}
where $\widetilde{S}(\tau, \xi)=S(\tau,\xi)-f(\tau,t)+(t,\xi)$.
The exponent has maximum at $\eta$ which is the Legendre image of $t$ and $\widetilde{S}(\tau, \eta)=0$.
Computing this expression in a neighborhood of the critical point $\eta$ we derive the following statement.

\begin{theorem}\label{con-2}
If $t$ is regular, the asymptotic character probability distribution is localized at point $\eta$ with a Gaussian distribution around this point. If we rescale random variable $\xi=\sum_a\xi_a\alpha_a$ near the critical point $\eta$ as $\xi_l=\eta_l+\sqrt{\epsilon}a_l$, then in the limit $\epsilon\to 0$ the random variable $a\in \hf^*_\RR\simeq \RR^r$\footnote{Here we use the basis of simple roots.} is distributed with the probability measure $p(a)da_1\dots da_n$ where $da$ is the Euclidean measure on $\RR^r$ and
\begin{equation}\label{B}
p(a,t)=\frac{\sqrt{detK}}{(2\pi)^{\frac{r}{2}}}e^{-\frac{1}{2}\sum_{bc}a_bK_{bc}a_c} .
\end{equation}
In other words, the probability measure $\{p_\lambda\}$ weakly converges to the probability measure $p(a)da$ 
\end{theorem}

Let $\Lambda$ be the weight lattice of the Lie algebra $\gf$ and $\epsilon_n$ be a positive sequence such that $\epsilon_n\to 0$ when $n\to \infty$. Define $\Lambda_{\epsilon_n}\subset \RR^r$  as the image of $\Lambda$ with respect to the embedding $\lambda=\sum_i\lambda_i\omega_i\mapsto\sum_i\epsilon_n\lambda_i\omega_i\in\mathbb{R}^r$.  Define the subset $D\subset \Lambda$ as the set of weights that appear in the decomposition of the tensor product (\ref{tp}) and the subset $D_{\epsilon_n}
\subset \Lambda_{\epsilon_n}\subset \RR^r$. Note that elements of $D$ have the form $\sum_kN_k\nu_k-\sum_{b=1}^ra_b\alpha_b$ where $a_b\geq 0$ are integers and $\alpha_a$ are simple roots. The weak convergence of measures in the context of theorem \ref{con-2} means the following. For any  bounded continuous function
$h$ on $\RR^r$, a sequence $\epsilon_n\to 0$, and the condition $\tau_k=\epsilon_nN_k$ are fixed for all $k$ in the 
tensor product 
\begin{equation}\label{cc}
\sum_{\lambda\in D} h(\sqrt{\epsilon_n}(\lambda-\frac{\eta}{\epsilon_n}))p^{(N)}_\lambda(t)\to \int_{\RR^r}h(a)p(a,t)da_1\dots da_r ,
\end{equation}

The extreme non regular case is when $t=0$. In this case the probability distribution 
is given by 
\begin{equation}
  \label{Plan}
   p_\lambda^{(N)}(0) =  \frac{m_{\lambda}(\{V_k\}, \{N_k\})dim(V_\lambda)}{\prod_k dim(V_{\nu_k})^{N_k}}.
\end{equation}

\begin{theorem}\label{con-3}
As $\epsilon\to 0$ the Plancherel measure (\ref{Plan}) weakly converges to the probability measure $p(a)da$ on $\hf^*_{\geq 0}\simeq \RR^r$ where $da$ is the Euclidean measure on $\RR^r$ and the density function is
\begin{equation}\label{Pl}
p(a)=\frac{\sqrt{det B}}{(2\pi)^{\frac{r}{2}}}\prod_{\alpha>0}\frac{(a,\alpha)^2}{(\rho,\alpha)}
e^{-\frac{1}{2}\sum_{b,c} a_bB_{bc}a_c} ,
\end{equation}
where random variables $\lambda$ and $a$ are related as $\lambda=\sqrt{\frac{1}{x\epsilon}}a$ and $x=\frac{1}{dim(\gf)}\sum_{k=1}^m \tau_k c_2(\nu_k)$. Here $c_2(\nu)$ is the value of the Casimir element on the irreducible
representation $V_{\nu}$.
\end{theorem}
This result generalizes slightly convergence results for $m=1$ from \cite{B}\cite{TZ}.

Now consider again the character measure on irreducible components of the tensor product
but assume that $t=\sqrt{\frac{\epsilon}{x}} u$ as $\epsilon\to 0$. In this case $p_\lambda^{(N)}(t)$
converges to a natural deformation of the measure (\ref{Pl}).

\begin{theorem} \label{con-4}As $\epsilon\to 0$, $t=\sqrt{\frac{\epsilon}{x}} u$ and $u$ is fixed, 
the character measure (\ref{stat}) weakly converges to the measure $p(b,u)db_1\dots db_r$ on $\RR^r_{\geq 0}$
with density
\begin{equation}\label{Plu}
p(b,u)=\frac{\sqrt{\det 
B}}{(2\pi)^{\frac{r}{2}}}
\frac{\prod_{\alpha\in \Delta_+}(b,\alpha)}{\prod_{\alpha\in \Delta_+}
(u,\alpha)}
\sum_{w\in W}(-1)^we^{(b,w(u))}e^{-\frac{(b,b)}{2}-\frac{(u,u)}{2}}
\end{equation}
Here random variables $\lambda$ and $b$ as related as $\lambda=\sqrt{\frac{1}{x\epsilon}}b$.
\end{theorem}

\subsection{} The paper is organized as follows. In section \ref{mult} we derive the asymptotic of the multiplicity function from the character identities. In section \ref{P} we derive the asymptotic distribution with respect to the character measure (\ref{stat}), assuming that $t$ is regular.
We also show there that the trivial representation is the most probable with respect to the Plancherel measure and derive the 
asymptotic statistics around it. Section \ref{ex-hook} contains an outline of the derivation of the asymptotic of the multiplicity 
function using the hook formula for $N$-th tensor power of the vector representation of $sl_{n+1}$ and 
the corresponding asymptotic distribution. In section \ref{ex-leg} we compute the same asymptotic from theorem \ref{con-1}. We derive intermediate asymptotic and outline the proof of Theorem \ref{con-3} in section \ref{Int}.
Nonlinear PDE's for the rate function are derived in section \ref{PDE}.
In the conclusion we outline some perspectives.

\subsection{Acknowledgments } We thank A. Nazarov and V. Serganova for stimulating discussions. We are 
grateful  for E. Feigin, V. Gorin, M. Walter and S. Zelditch for important remarks and for 
pointing us to recent works on the subject. 
This work was supported by RSF-18-11-00-297. The work
of NR was partly supported by the grant NSF DMS-1601947.

\section{ Plancherel and character measures as density matrices}

Let $W$ be a finite dimensional representation of a simple Lie algebra $\gf$. Let 
\[
W\simeq\oplus V_i^{\oplus m_i}
\]
be the decomposition of $W$ into irreducibles where $m_i$ is the multiplicity of $V_i$.
By analogy with the usual Plancherel measure for left regular representations we will say that 
{\it the Plancherel measure on} $W$ is the probability measure on the set of irreducible components
of $W$ with the probability of $V_i$ being
\[
p_i=\frac{dim(V_i)m_i}{dim(W)} .
\]

This probability measure can be regarded as a quantum probabilistic state on $End(W)$
which assigns to an operator $A:W\to W$  its expectation value $tr(A)/dim(W)$.
In other words, this is the state with the density matrix
\[
\rho=\frac{I}{dim(W)} ,
\]
where $I:W\to W$ is the identity operator acting in $W$.
This density matrix can be interpreted as the quantum equilibrium  state 
\[
\rho=\frac{e^{-\frac{H}{T}}}{tr(e^{-\frac{H}{T}})}
\]
with zero Hamiltonian,  $H=0$, i.e. the topological quantum mechanics 
with the space of states $W$.

The character measure (\ref{stat}) can be interpreted in a similar way.
It corresponds to the Hamiltonian being semisimple element of $\gf$, i.e.  
by a conjugation such an $H$ can be brought to an element of the Cartan subalgebra.
In notations used in the introduction $t=-H/T$.
         
\section{Derivation of the asymptotic of multiplicities}\label{mult}
 
In this section we give arguments justifying conjecture \ref{con-1} for the asymptotic of the multiplicity function for tensor product of representations and show how to obtain exponential and sub exponential terms of this asymptotic from their characters.
From now on we assume that $\lambda$ is regular.

\subsection{Large deviation rate function $S(\tau,\xi)$}
\label{11}
We expect\footnote{This is proven in \cite{TZ} for $m=1$. The proof is based on the application of the steepest descent method 
and on the orthogonality of characters of irreducible representations of corresponding compact  Lie group. For $m>1$ the proof is
completely similar. } that in the limit  $N=\tau/\epsilon$ and $\lambda=\xi/\epsilon,\;\;\epsilon\to 0$  with $\nu_k$ and $x$ fixed, the  multiplicity function $m_\lambda^N$ has the following asymptotic:
\begin{equation}
\label{mas}
    m_\lambda^N=e^{\frac{1}{\epsilon}S(\tau,\xi)}\mu(\tau,\xi)(1+O(\epsilon)).
\end{equation}
Here $S(\tau,\xi)$ is the large deviation rate function and $\mu(\tau,\xi)$ is a function of $\tau,\xi$ which is 
proportional to a power of $\epsilon$. We will now derive explicit expressions for these functions via character identities. Note that here $\lambda$ is in the main Weyl chamber and therefore $\xi\in \hf^*_{\geq 0}$
where $\geq 0$ means $\xi=\sum_{a=1}^r\xi_a\alpha_a$ has nonnegative coordinates $\xi_a$.

From the tensor product decomposition we have:
\begin{equation}\label{chid}
        \prod_k \chi_{\nu_k}(e^t)^{N_k}=\sum_{\lambda\in D} m_{\lambda}(\{V_k\}, \{N_k\})\chi_\lambda(e^t) .
    \end{equation}
where $\chi_{\lambda}(e^t)=tr_{V_\lambda}(e^t)$ is the character of module $V_\lambda$
evaluated at $e^t\in G$ and $D\subset \Lambda$ is as in (\ref{cc}).

Now let us substitute  (\ref{mas}) into (\ref{chid}) and replace the summation over $\lambda$
by an integral over $\xi$:
\begin{equation}
\label{sumint}
e^{\frac{f(\tau,t)}{\epsilon}}=\sum_{\lambda} m_{\lambda}^{N}\chi_{\lambda}(e^t)\simeq
\epsilon^{-r}\int_{D_0}
e^{\frac{1}{\epsilon}S(\tau,\xi)}\mu(\tau,\xi)
\chi_{\frac{\xi}{\epsilon}}(e^t)d
\xi,
\end{equation}
Here $r$  is the rank of Lie algebra $\gf$, $D_0\subset \RR^r$ is the limiting set of $D_\epsilon$ as $\epsilon\to 0$.
$D_0$ is a convex polytope  in the positive chamber $\RR^r_{\geq 0}$ with $\sum_k\tau_k\nu_k$ being the extreme point 
which is farthest from the origin. The measure $d\xi$ is the Euclidean measure $d\xi_1\cdots d\xi_r$.

Let us use the Weyl character formula 
\[
\chi_{\lambda}(e^t)=\frac{\sum_we^{(w(\lambda+\rho),t)}(-1)^{l(w)}}{\Delta(t)} ,   
\]
where $\rho=\frac{1}{2}\sum_{\alpha\in\Delta^+}\alpha,$ and $l(w)$ is the length of $w$ (in terms of simple reflections). When $t$ is regular and is in the positive Weyl chamber and  $\lambda=\xi/\epsilon$ where $\xi$ is fixed and is also in the positive Weyl chamber, and $\epsilon\to 0$, the term with $w=e$ gives the leading asymptotic and we have:
\begin{equation}
\label{smallchi}
    \chi_{\frac{\xi}{\epsilon}}(e^t)=e^{\frac{1}{\epsilon}(\xi,t)}\frac{e^{(\rho,t)}}{\Delta(t)}(1+o(1)).
\end{equation}
From now on we will identify $\hf^*_{\geq 0}$ with $\RR^r_{\geq 0}$ using the basis of simple roots.

Substituting (\ref{smallchi}) into (\ref{sumint}) we obtain:
\begin{equation}
\label{int2}
   e^{\frac{f(\tau,t)}{\epsilon}}\simeq\epsilon^{-r} \int_{D_0}
e^{\frac{1}{\epsilon}S(\tau,\xi)}\mu(\tau,\xi)
\chi_{\frac{\xi}{\epsilon}}(x)d\xi\simeq
\epsilon^{-r}\int_{\RR^r_{\geq 0}}
e^{\frac{1}{\epsilon}(S(\tau,\xi)+(\xi,t))}\mu(\tau,\xi)
\frac{e^{(\rho,t)}}{\Delta(t)}d\xi.
\end{equation}
where $d\xi=d\xi_1\dots d\xi_r$.

Assume that $S(\tau, \xi)$ is strictly concave in $\xi$. This implies that 
$S(\tau,\xi)+(\xi,t)$  has unique maximum at some point $\eta$. Then the leading 
contribution to the integral comes from the vicinity of $\eta$. It is given by the Gaussian integral over
$a=(\xi- \eta)/\sqrt{\epsilon}$ and determined by the second Taylor coefficient of $S(\tau,\xi)$:
\begin{equation}
\label{long}
e^{\frac{f(\tau,t)}{\epsilon}}=e^{\frac{1}{\epsilon}(S(\tau,\eta)+(t,\eta))}\mu(\tau,\eta)\epsilon^{-r}\frac{e^{(\rho,t)}}{\Delta(t)}
\epsilon^{\frac{r}{2}}
\int_{\mathbb{R}^r}e^{\frac{1}{2}(a,S^{(2)}a)}d^ra
\end{equation}
where $S^{(2)}_{ab}=
\frac{\partial^2}{\partial \xi_a \partial \xi_b}S(\tau,\xi)|_{\eta} $.

Comparing most singular terms in this equation we can identify $S(\tau,\xi)$  with the Legendre transform of $ f(\tau,y)$ in $y$: 
\begin{equation}
\label{s}
    S(\tau,\xi)= \min_yf(\tau, y)-(y,\xi))=f(\tau,x)-(x,\xi),
\end{equation}
Because the characters are strictly convex, the function $f(\tau, x)$ is strictly convex in $x$. This implies
that $S(\tau, \xi)$ is strictly concave in $x$, which agrees with the assumption made earlier. 
Note that for the inverse Legendre transform we have
\begin{equation}
\label{leg}
    f(\tau,x)=\max_\xi(S(\tau,\xi)+(x,\xi))=S(\tau,\eta)+(x,\eta),
\end{equation}
where $x$ and $\eta$ are related as 
\begin{equation}
\label{max2}
    \tau\frac{\partial}{\partial x_a}\ln \chi_\omega(x)=\sum_b B_{ab}\eta_b,
\end{equation}
\begin{equation}
\label{max1}
    \frac{\partial S(\tau,\xi)}{\partial \xi_a}|_{\xi=\eta}=-\sum_bB_{ab}x_b.
\end{equation}
Here $B_{ab}$ is the symmetrized Cartan matrix and $(\eta,\xi)=\sum_a \eta_a B_{ab}\xi_b$.

Since $f(\tau,x))$ is a convex function of $x$ with minimum at $x=0$, the function $S(\tau,\xi)$  defined in (\ref{s}) is a concave function of $\xi$  with the unique maximum at $\xi=0$.

\subsection{Subexponential terms}\label{12}
Now after deriving most singular exponential factors, the equation (\ref{long}) is reduced to
\begin{equation}
\label{long2}
1=\mu(\tau,\eta)\frac{e^{(\rho,t)}}{\Delta(t)}\epsilon^{-\frac{r}{2}}\int_{\mathbb{R}^r}e^{\frac{1}{2}(a,S^{(2)}a)}d^ra.\end{equation}
Taking into account that the Gaussian integral can be computed as
\[
    \int_{\mathbb{R}^r}e^{-\frac{1}{2}(a,Ka)}d^ra=
    (2\pi)^{r/2}\frac{1}{\sqrt{\det K}} .
\]
the equation(\ref{long2}) gives:
\[
\mu(\tau,\eta)\frac{e^{(\rho,x)}}{\Delta(x)}\epsilon^{-\frac{r}{2}}(2\pi)^{r/2}\frac{1}{\sqrt{\det K}}=1 .
\]
where $K_{ab}=-\frac{\partial^2S}{\partial \xi_a \partial \xi_b}(\xi)|_{\xi=\eta}$ and $\xi_0$ is the Legendre image of $x$.

Therefore, we derived the unknown function $\mu(\tau,\xi)$ as
\[
 \mu(\tau,\eta)= \epsilon^{r/2}  \frac{\sqrt{\det K}}{(2\pi)^{r/2}}\Delta(x)
 e^{-(\rho,x)} .
\]
where $\eta$ and $x$ are related as in (\ref{max2}).

To calculate  partial $K$ explicitly we will use (\ref{max1}) and (\ref{max2}). Differentiating (\ref{max1}) with respect to $x_d$ we get
\[
   \frac{\partial^2S}{\partial \xi_a \partial \xi_c}(\xi)|_{\xi=\eta} \frac{\partial(\eta)_c}{\partial x_d}=-B_{ad} ,
\]
then
\[
   \frac{\partial^2S}{\partial \xi_a \partial \xi_c}(\xi)|_{\xi=\eta}=-B_{ad}\frac{\partial x_d}{\partial \eta_c}.
\]
To obtain $\frac{\partial x_d}{\partial \eta_c}$ we will differentiate (\ref{max2}) with respect to $\eta_c$: 
\[
     \frac{\partial^2}{\partial x_a \partial x_b}f(\tau,x)\frac{\partial x_b}{\partial \eta_c}=B_{ac} .
\]
Here $x$ is regarded as the Legendre image of $\eta$.
Denote
\begin{equation}
\label{b}
    D_{ab}= \frac{\partial^2}{\partial x_a \partial x_b}f(\tau,x),
\end{equation}
then
\[
D_{ab}\frac{\partial x_b}{\partial \xi_c}=B_{ac} ,
\]
and therefore
\[
    \frac{\partial x_b}{\partial \xi_c}=(D^{-1})_{ba}B_{ac} .
\]
Finally, we derive the formula for $K$ in terms of $f$:
\[
   K_{ab}=B_{ac}(D^{-1})_{cd}B_{db} .
\]

\section{Probability distributions}\label{P}
Here we will derive probability distributions for two cases, for the character measure and
for the Plancherel measure, i.e. we will outline proofs of Theorems \ref{con-2} and \ref{con-3}.

\subsection{Character distribution.}
Here we outline the proof of Theorem \ref{con-2}.
The function
\[
\widetilde{S}(\tau,\xi)=S(\tau,\xi)-f(\tau,t)+(t,\xi)
\]
in the formula (\ref{A}) is strictly concave in $\xi$ and therefore it has unique 
maximum at some point $\eta$, which is the Legendre image of $t$ with respect to the function $S(\tau,\xi)$.  Because $S(\tau,\xi)$ is invariant with respect to the action of the 
Weyl group, the point $\eta$ is regular if $t$ is regular. In the vicinity of $\eta$ we have
\[
\widetilde{S}(\tau,\xi)=-\frac{\epsilon}{2} \sum_{a,b} a_aD_{ab}a_b+O(\epsilon^{3/2}),
\]
where $a_a=(\xi_a-\eta_a)/\sqrt{\epsilon}$ are the coordinates in the basis of simple roots. Taking this into account and passing to the limit $\epsilon\to 0$
in (\ref{A}) we obtain the Gaussian distribution (\ref{B}).

\subsection{Plancherel distribution}\label{planch}
Here we outline the proof of Theorem \ref{con-3} under the assumption that the asymptotic of multiplicities 
is uniform in $\xi$ in $D_0$ including boundary strata\footnote{For a rigorous proof see \cite{B} and \cite{TZ}}. 
In other words the asymptotic along the 
boundary can be obtained by taking corresponding limit in $\xi$. 

For the Plancherel probability distribution we have the following asymptotic when $\epsilon\to 0$ and $\xi=\epsilon \lambda$ is finite  and regular:
\begin{equation}\label{Plan-prob-1}
p_\lambda^{(N)}(t)=\epsilon^{\frac{r}{2}+|\Delta_+|}\frac{\sqrt{detK}\prod_{\alpha>0}(\xi,\alpha)}{(2\pi)^{\frac{r}{2}}\prod_{\alpha>0}(\rho,\alpha)}\Delta(x)e^{-(\rho,x)}
e^{\frac{1}{\epsilon}S(\tau,\xi)}(1+O(\epsilon)) .
\end{equation}
We used the Weyl formula for the dimension of $V_\lambda$:
\[
dim(V_\lambda)=\prod_{\alpha\in \Delta_+}\frac{ (\lambda+\rho,\alpha)}{(\rho, \alpha)}
\]
The function $S(\tau,\xi)$ attains its maximum at $\xi=0$. So, we shall find the asymptotic of
the formula (\ref{Plan-prob-1}) near this point. The convexity of the function $f$ implies that for 
small $\xi$, its Legendre dual $x$ is also small, therefore for small $\xi$
\[
    \sum_b\frac{\partial^2 f(\tau,0)}{\partial x_a\partial x_b}x_b=
    \sum_b B_{ab}\xi_b +O(\xi^2).
\]
Now let us compute the matrix of second derivative
\begin{equation}
\label{l}
\frac{\partial^2 f(\tau,0)}{\partial x_a\partial x_b}=
\frac{\partial^2}{\partial x_a\partial x_b}
\sum_k\tau_k\ln(tr_{V_{\nu_k}}(e^{\sum_ax_aH_a}))|_{x=0} .
\end{equation}
Here $\{H_a\}$ is a basis of simple roots in the Cartan subalgebra.

\begin{lemma}\label{L1}
The following equality holds:
\[
\frac{\partial^2}{\partial x_a\partial x_b}(\ln tr_{V_{\nu}}(e^{x}))|_{x=0}=\frac{c_2(\nu)}{dim(\mathfrak{g})}B_{ab} ,
\]
where $B_{ab}$ is the symmetrized Cartan matrix and $c_2$ is the Casimir element (see below).
\end{lemma}
\begin{proof}
\begin{equation*}
\frac{\partial^2}{\partial x_a\partial x_b}(\ln tr_{V_{\nu}}(e^{x}))|_{x=0}=\frac{\partial}{\partial x_a}\frac{\frac{\partial}{\partial x_b} tr_{V_{\nu}}(e^{x})}{tr_{V_{\nu}}(e^{x})}|_{x=0}=
\end{equation*}
\begin{equation*}
=\frac{\frac{\partial^2}{\partial x_a\partial x_b} 
\left(tr_{V_{\nu}}(e^{x})\right)tr_{V_{\nu}}(e^{x})-\frac{\partial}{\partial x_a}\left( tr_{V_{\nu}}(e^{x})\right)\frac{\partial}{\partial x_b}\left( tr_{V_{\nu}}(e^{x})\right)}{\left(tr_{V_{\nu}}(e^{x})\right)^2}|_{x=0}=
\end{equation*}
\begin{equation*}
    =\frac{tr_{V_{\nu}}\left(H_aH_b e^{x}\right)tr_{V_{\nu}}\left(e^{x}\right)-tr_{V_{\nu}}\left(H_a e^{x}\right)tr_{V_{\nu}}\left(H_b e^{x}\right)}{\left(tr_{V_{\nu}}\left(e^{x}\right)\right)^2}|_{x=0}=
    \end{equation*}
    \begin{equation}
    =\frac{tr_{V_{\nu}}\left(H_a H_b \right)tr_{V_{\nu}}(1)}{\left(tr_{V_{\nu}}(1)\right)^2}=
    \frac{tr_{V_{\nu}}\left(H_a H_b \right)}{\dim V_{\nu}} .
\end{equation}
To find the numerator choose a basis $e_i$ in $\mathfrak{g}$. If $(\;.\;,\;.\;)$ is the Killing form and
\[
   Q_{ij}=(e_i,e_j)=tr_{ad}(e_ie_j),\;\;\; Q^{ij}Q_{jk}=\delta_k^i .
\]
then for the Casimir element we have
\[
 c_2=Q^{ij}e_ie_j   , 
\]
By Schur's lemma
\begin{equation}
    tr_{V_{\nu}}\left(e_ie_j\right)=\tilde{x}_{\nu}Q_{ij} .
\end{equation}
for some $\tilde{x}_{\nu}$. This constant is easy to find taking the trace: 
\[
  tr_{ V_{\nu}}\left(B^{ij}e_ie_j\right)=c_2(\nu)tr_{V_{\nu}}(1)= c_2(\nu)\dim (V_{\nu}) ,  
\]
which gives
\[
   \tilde{x}_{\nu}=
   \frac{c_2(\nu)dim (V_{\nu})}{\dim{\mathfrak{g}}} ,
\]
Assume that  ${H_a}$ is part of the basis $e_i$, then 
\[
    tr_{V_{\nu}}\left(H_a H_b\right)=\frac{c_2(\nu)dim (V_{\nu})}{\dim{\mathfrak{g}}}tr_{ad}\left(H_aH_b\right)=
    \frac{c_2(\nu)dim (V_{\nu})}{\dim{\mathfrak{g}}}d_aC_{ab}.
\]
Finally,
\[
\frac{\partial^2}{\partial x_a\partial x_b}(\ln tr_{V_{\nu}}(e^{x}))|_{x=0}=
\frac{c_2(\nu)}{\dim{\mathfrak{g}}}d_aC_{ab}.
\]
\end{proof}

Denote $x=\frac{1}{dim(\gf)}\sum_{k=1}^m \tau_k c_2(\nu_k)$. As a consequence of the lemma above and 
the asymptotical  formula for $x_a$ in terms of $\xi_a$ we have
\[
x_a=\frac{\xi_a}{x}+O(\xi^2) .
\]
Thus, for small $\xi$ we have
\[
\prod_{\alpha>0} (x,\alpha)=x^{-|\Delta_+|}\Delta(\xi)(1+o(1)) .
\]

Recall that the matrix $K$ is defined as
\[
K_{ab}=-\frac{\pa^2S(\tau,\xi)}{\pa \xi_a\pa \xi_b}=\sum_{bd}B_{ad}(D^{-1})_{dc}B_{cb} .
\]
Lemma \ref{L1} implies that 
\[
D_{ab}|_{\xi=0}=\frac{\pa^2 f(\tau, y)}{\pa y_a\pa y_b}|_{y=0}=xB_{ab} .
\]
Thus, $K=B/x$ and for small $\xi$ we have
\[
S(\tau, \xi)=S(\tau,0)-\frac{1}{2x} \xi_aB_{ab}\xi_b+O(\xi^3) .
\]

Rescaling random variable $\xi$ as $\xi_b=\sqrt{\epsilon x} a_b$ we obtain the following asymptotic
for the probabilities in Plancherel distribution:
\begin{equation}\label{plas}
p_\lambda^{(N)}=\left(\frac{\epsilon}{x}\right)^{\frac{r}{2}} \frac{\sqrt{detB}}{(2\pi)^{\frac{r}{2}}}\prod_{\alpha>0}\frac{(a,\alpha)^2}{(\rho,\alpha)}
e^{-\frac{1}{2}\sum_{b,c} a_bB_{bc}a_c}(1+o(1)) .
\end{equation}

Let $h$ be a continuous bounded function on $\RR^r_{\geq 0}$. The asymptotic (\ref{plas}) implies
the convergence 
\[
\sum_\lambda h(\sqrt{\frac{\epsilon}{x}}\lambda)p_\lambda^{(N)}\to \int_{\RR^r_{\geq 0}}h(a)p(a)da_1\dots da_r ,
\]
where the density function $p(a)$ is given in (\ref{Pl}). In other words, the Plancherel measure  weakly converges to $p(a)da_1\dots da_r$. 

\section{The asymptotic for $\gf=sl_{n+1}$, $m=1$, $\nu=\omega_1$ via the hook length formula.}
\label{ex-hook} 
In this section we recall the derivation of the pointwise asymptotic of multiplicities  and of the Plancherel probability measure 
using the Schur-Weyl duality and the hook length formula from \cite{K1}.

\subsection{The asymptotic of multiplicity}
Let us first review known results on multiplicities of irreducible $sl_{n+1}$ modules 
in the $N$-th tensor power of the vector representation. 
Due to the Schur-Weyl duality we have
\[
    (\mathbb{C}^{n+1})^{\otimes N}\simeq \oplus_\lambda W_\lambda^{(N)}\otimes V_\lambda^{(n+1)},
 \]
where $V_\lambda^{(n+1)}$ is the irreducible $sl_{n+1}$ module with the highest weight $\lambda$ corresponding to the partition $l_1\geq\dots \geq\l_{n+1}\geq 0$, $\sum_{i=1}^{n+1}l_i=N$ and in the simple root basis $\lambda_a=l_a-l_{a+1}$. The space $W_{\lambda}^{(N)}$ is an irreducible $S_N$ module corresponding to the partition $l$; $m_\lambda^{(N)}=dim(W_{\lambda}^{(N)})$ is the multiplicity of $V_\lambda \hookrightarrow(\mathbb{C}^{n+1})^{\otimes N}$. Here the Lie algebra $sl_{n+1}$ acts diagonally and  the symmetric group $S_N$ permutes the factors.

The multiplicity function $m_\lambda^{(N)}$ (or the dimension of $W_{\lambda}^{(N)}$) is determined by the hook length formula, see, for example \cite{K2}:
\[
m_\lambda^{(N)}=N!\frac{\prod_{i<j}\left(l_i-l_j-i+j\right)}
{\prod_{i=1  }^{n+1}\left(l_i+n+1-i\right)!}    .
\]
Using the Stirling formula 
\[
    N!=\sqrt{2\pi N}e^{N\ln N-N}\left(1+O\left(\frac{1}{N}\right)\right),
\]
we obtain the following asymptotics for multiplicities for large $N$ and $l_i$:
\[
m_\lambda^{(N)}=
\frac{\sqrt{2\pi N} e^{N\ln N-N }\prod_{i<j}\left(l_i-l_j\right)}{(\sqrt{2\pi})^{n+1}\prod_{i=1  }^{n+1}l_i^{n+1-i+1/2}e^{l_i\ln l_i-l_i}}
\left(1+O\left(\frac{1}{N}\right)\right).
\]
Now, assume $N=\tau/\epsilon,\;\;l_i=\sigma_i/\epsilon$, where $\tau,\sigma_i$ are finite and $\epsilon\to 0$. Note that $\tau,\sigma_i$ satisfy the condition $\sum_{i=1}^{n+1}\sigma_i=\tau,\;\;\;\tau\geq\sigma_1\geq\dots\geq\sigma_{n+1}\geq 0$.

Then in the limit $\epsilon\to 0$ we obtain the asymptotic of the multiplicity function:
\begin{equation}
\label{multg}
m_\lambda^{(N)}=
\left(\frac{\epsilon}{2\pi}\right)^{\frac{n}{2}}\tau^{\frac{1}{2}}
\prod_{i<j}\left(\sigma_i-\sigma_j\right)\prod_{i=1  }^{n+1}\sigma_i^{-n+i-3/2}
e^{\frac{1}{\epsilon} S(\tau,\sigma)}
\left(1+O(\epsilon)\right),
\end{equation}
where
\begin{equation}
\label{slns}
    S(\tau,\sigma)=\tau\ln\tau -\sum_{i=1}^{n+1}\sigma_i\ln\sigma_i.
\end{equation}
In the section \ref{ex-leg} we will show that this formula matches the one from Theorem 1. 

\subsection{The asymptotic of the Plancherel probability distribution.}

Now let us find the asymptotical probability distribution for the Plancherel probability measure which was
first studied by Kerov in \cite{K1}. 
\[
p_{\lambda}^{(N)}=\frac{  m_{\lambda}^{(N)}\dim V^{\lambda}}{(n+1)^N} ,
\]
\subsubsection{}Dimensions of irreducible $sl_{n+1}$-modules are given by the Weyl formula 
\[
  \dim V^{\lambda}=\frac{\prod_{i\leq j}(l_i-l_j)}{\prod_{k=1}^{n+1}k!}.
\]

First, let us find the asymptotic of $p_\lambda$ when $N=\tau/\epsilon,\;\;l_i=\sigma_i/\epsilon$, $\epsilon\to 0$ and  $\tau,\sigma_i$ remain finite with $\sum_{i=1}^{n+1}\sigma_i=\tau,\;\;\;\tau\geq\sigma_1\geq\dots\geq\sigma_{n+1}\geq 0$.
When $\epsilon\to 0$ we have
\[
\frac{\prod_{i\leq j}(l_i-l_j)}{\prod_{k=1}^{n+1}k!}\simeq  
  \epsilon^{\frac{(n+1)^2-n-1}{2}}\frac{\prod_{i\leq j}(\sigma_i-\sigma_j)}{\prod_{k=1}^{n+1}k!} .
\]
Combining this with the asymptotic of the multiplicity we obtain the following pointwise asymptotic of $p_\lambda$:
\[
p_\lambda=\left(\frac{\epsilon}{2\pi}\right)^{\frac{n^2+2n}{2}}
\frac{\prod_{i<j}\left(\sigma_i-\sigma_j\right)^2}{\prod_{k=1}^{n}k!}\prod_{i=1}^{n+1}\sigma_i^{-n+i-3/2}
e^{\frac{1}{\epsilon} (S(\tau,\sigma)-\tau\ln(n+1))}
\left(1+O\mathcal{O}(\epsilon)\right),
\]

\subsubsection{}Now, let us first find the maximum of the large deviation rate function $S(\tau,\sigma)$ on the hypersurface $\sum_{i=1}^{n+1}\sigma_i=\tau$. 
For this we can use the method of Lagrange multipliers and consider
\[
S_\alpha(\tau,\sigma)=S(\tau,\sigma)+\alpha (\sum_{i=1}^{n+1}\sigma_i-\tau) .
\]
Critical points of $S(\tau,\sigma)$ are solutions to $\frac{\pa S_\alpha}{\pa \sigma_i}=0$ and $\frac{\pa S_\alpha}{\pa \alpha}=0$
which gives:
\begin{equation}\label{constr}
-(\ln\sigma_i+1)+\alpha=0, \ \  \sum_{i=1}^{n+1}\sigma_i=\tau
\end{equation}
This system has unique solution 
\[
    \sigma_i=\frac{\tau}{n+1}.
\]
Now let us study the behaviour of $m_\lambda^N$ in the vicinity of this critical point. For this rescale
random variables $\sigma_i$ as:
\begin{equation}
\label{fluct}
  \sigma_i=  \frac{\tau}{n+1}+\sqrt{\epsilon}a_i .
\end{equation}
Since $\sigma_i$ are constrained  (\ref{constr}), we should have $\sum_{i=1}^{n+1}a_i=0$
and since $\sigma_i\geq \sigma_{i+1}\geq 0$ we should have $a_i\geq a_{i+1}$.
Expanding $S(\tau,\sigma)$ in the Taylor series around $\sigma_i=\tau/(n+1)$ we have
\[
 S(\tau,\sigma)=S(\tau,\tau/(n+1))-\frac{n+1}{\tau} \sum_{i=1}^{n+1}\frac{\epsilon a_i^2}{2} +O(\epsilon^{3/2}) .
\]

For the pointwise asymptotic of the asymptotic Plancherel probability distribution 
in the vicinity of the critical point we obtain
\begin{equation}
   p_\lambda^{(N)}=
\left(\frac{\epsilon}{2\pi}\right)^{\frac{n}{2}}\frac{1}{1!\cdot 2!\cdot\dots n!}
\tau^{\frac{-(n+1)^2+1}{2}}
\left(n+1\right)^{\frac{(n+1)^2}{2}}
\prod_{i<j}\left(a_i-a_j\right)^2
e^{-\frac{1}{2} \frac{n+1}{\tau}\sum_{i+1}^{n+1}a_i^2}
\left(1+O(\epsilon)\right).  
\end{equation}
This implies  that the probability measure $p_\lambda$ converges weakly to the probability distribution of the hyperplane $\sum_i a_i=0$ in $\RR^{n+1}$ with the density function
\[
p(a)=\left(\frac{1}{2\pi}\right)^{\frac{n}{2}}\frac{1}{1!\cdot 2!\cdot\dots n!}
\tau^{\frac{-(n+1)^2+1}{2}}
\left(n+1\right)^{\frac{(n+1)^2}{2}}
\prod_{i<j}\left(a_i-a_j\right)^2
e^{-\frac{1}{2} \frac{n+1}{\tau}\sum_{i+1}^{n+1}a_i^2}
\left(1+O(\epsilon)\right)
\]
This is exactly the result of \cite{K1}.

\section{The asymptotic of multiplicities  for $\gf=sl_{n+1}$, $m=1, \nu=\omega_1$  via character identities.}\label{ex-leg}

Here we will show that for $N$-the tensor power of the vector representation of $sl_{n+1}$ the asymptotic of the multiplicity $ m_{\lambda}^N$ derived from Theorem 1 
\[
    m_{\lambda}^N=
    \epsilon^{\frac{r}{2}}\frac{\sqrt{detK}}{(2\pi)^{\frac{r}{2}}}\Delta(x)e^{-(\rho,x)}
e^{\frac{1}{\epsilon}S(\tau,\xi)}\left(1+O(\epsilon)\right)
\]
coincides with  the asymptotic obtained from the hook length formula.

\subsection{The large deviation rate function $S(\tau,\xi)$} For $sl_{n+1}$ the simple roots are $\alpha_i=e_i-e_{i+1},\;\;i=1\dots n$. The  first fundamental weight is $\omega_1=\frac{1}{n+1}\left(ne_1-e_2-e_3-\dots-e_n\right)$. Weights of the first fundamental representation are given by $$\mu_k=\frac{1}{n+1}\left(e_1-e_2-\dots-e_{k-1}+ne_k-e_{k+1}-\dots-e_n\right),\;\; k=1\dots n+1 ,$$ where $\mu_1= \omega_1$ is the highest weight.

The character of the first fundamental module:
\[
    \chi_{\omega_1}(e^y)=e^{(y,\mu_1)}+ e^{(y,\mu_2)}
    +\dots+e^{(y,\mu_{n+1})}=e^{y_1}+e^{-y_2+y_3}+e^{-y_3+y_4}+\dots+e^{-y_{n-1}+x_n}+e^{-y_n} ,
\]
where $(y,\mu_i)$ is the scalar product in $
\mathbb{R}^{n+1}$ and $y=y_1\alpha_1+y_2\alpha_2+\dots+y_n\alpha_n$.

According to the proposed method, the large deviation rate function for the first fundamental representation is given by the following expression:
\begin{equation}
\label{sw1}
    S(\tau,\xi)=\tau\ln(\chi_{\omega_1}(x))-B_{ab}{x}_a,\xi_b,
    \end{equation}
where $x$ is
determined from the system of equations (\ref{max2}):
\begin{equation}
\label{max5}
    \tau\frac{\partial}{\partial x_a}\ln \chi_{\omega_1}(x)=\sum_b B_{ab}\xi_b.
\end{equation}
\begin{lemma}
The solution of system (\ref{max5}) is determined by 
\begin{equation*}
    \begin{cases}
   x_1=\ln\left(\frac{\chi}{\tau}\right)+\ln\left(\frac{\tau}{n+1}+\xi_1\right)\\
   \dots\\
   x_{i+1}=(i+1)\ln\left(\frac{\chi}{\tau}\right)+\ln\left(\frac{\tau}{n+1}+\xi_1\right)
   +\ln\left(\frac{\tau}{n+1}-\xi_1+\xi_2\right)+\dots+
   \ln\left(\frac{\tau}{n+1}-\xi_i+\xi_{i+1}\right)\\
   \dots\\
   x_n=-\ln\left(\frac{\chi}{\tau}\right)-\ln\left(\frac{\tau}{n+1}-\xi_n\right)
\end{cases} ,
\end{equation*}
where $\chi=\chi_{\omega_1}$ denotes the character of the first fundamental module.
\begin{proof}Firstly, we write the equations of (\ref{max5}) explicitly:
\begin{equation}
\label{slnsys}
    \begin{cases}
    \frac{\tau}{\chi}\left(e^{x_1}-e^{-x_1+x_2}\right)=2\xi_1-\xi_2\\
    \frac{\tau}{\chi}\left(e^{-x_{i-1}+x_{i}}-e^{-x_i+x_{i+1}}\right)=-\xi_{i-1}+2\xi_i-\xi_{i+1},\;\;i=2\dots n-1\\
    \frac{\tau}{\chi}\left(e^{-x_{n-1}+x_n}-e^{-x_n}\right)=-\xi_{n-1}+2\xi_n
\end{cases} .
\end{equation}
To solve (\ref{slnsys}) let us introduce the variables $z_i=e^{x_i}$. So (\ref{slnsys}) can be written as:
\[
    \begin{cases}
    z_1-z_1^{-1}z_2=\frac{\chi}{\tau}(2\xi_1-\xi_2)\\
    z_{i-1}^{-1}z_i-z_i^{-1}z_{i+1}=\frac{\chi}{\tau}(-\xi_{i-1}+2\xi_i-\xi_{i+1}),\;\;i=2\dots n-1\\
      z_{n-1}^{-1}z_n-z_n^{-1}=\frac{\chi}{\tau}(-\xi_{n-1}+2\xi_n)
\end{cases} .
\]
From the first equation we can express $z_1^{-1}z_2$ in terms of $z_1$ and then substitute it into the next equation. Step by step, we can express $z_i^{-1}z_{i+1}$ in terms of $z_1$:
\begin{equation}
\label{slnsys2}
    \begin{cases}
   z_1^{-1}z_2=z_1-\frac{\chi}{\tau}(2\xi_1-\xi_2)\\
    z_i^{-1}z_{i+1}=z_1-\frac{\chi}{\tau}(\xi_{1}+\xi_i-\xi_{i+1}),\;\;i=2\dots n-1\\
      z_n^{-1}=z_1-\frac{\chi}{\tau}(\xi_{1}+\xi_n)
\end{cases} .
\end{equation}
We then need to express $z_1$ in terms of $\chi$. This can be obtained by noting that
\begin{equation}\label{chiz}
    \chi=z_1+z_1^{-1}z_2+z_2^{-1}z_3+
    \dots+z_{n-1}^{-1}z_n+z_n^{-1} .
\end{equation}
Substiting $z_i^{-1}z_{i+1}$ from (\ref{slnsys2})  into (\ref{chiz}) we get
\[
    z_1=\frac{\chi}{\tau}\left(\frac{\tau}{n+1}+\xi_1\right) .
\]
Now, we can solve (\ref{slnsys2}) in terms of $\chi$ and ${\xi_i}$:
\begin{equation}
\label{slnz}
    \begin{cases}
   z_1=\frac{\chi}{\tau}\left(\frac{\tau}{n+1}+\xi_1\right)\\
   z_{i+1}=\left(\frac{\chi}{\tau}\right)^{i+1}\left(\frac{\tau}{n+1}+\xi_1\right)
   \left(\frac{\tau}{n+1}-\xi_1+\xi_2\right)\dots
   \left(\frac{\tau}{n+1}-\xi_i+\xi_{i+1}\right)\\
   z_n^{-1}=\frac{\chi}{\tau}\left(\frac{\tau}{n+1}-\xi_n\right)
\end{cases} .
\end{equation}
Finally, we can solve (\ref{slnsys})in terms of $\chi$ and ${\xi_i}$:
\begin{equation}
\label{slnz}
    \begin{cases}
   x_1=\ln\left(\frac{\chi}{\tau}\right)+\ln\left(\frac{\tau}{n+1}+\xi_1\right)\\
   x_{i+1}=(i+1)\ln\left(\frac{\chi}{\tau}\right)+\ln\left(\frac{\tau}{n+1}+\xi_1\right)
   +\ln\left(\frac{\tau}{n+1}-\xi_1+\xi_2\right)+\dots+
   \ln\left(\frac{\tau}{n+1}-\xi_i+\xi_{i+1}\right)\\
   x_n=-\ln\left(\frac{\chi}{\tau}\right)-\ln\left(\frac{\tau}{n+1}-\xi_n\right)
\end{cases} .
\end{equation}
\end{proof}
\end{lemma}
\begin{lemma}The large deviation rate function (\ref{sw1}) is given by
\begin{equation}
\label{slns1}
    S(\tau,\sigma)=\tau\ln\tau -\sum_{i=1}^{n+1}\sigma_i\ln\sigma_i.
\end{equation}
\begin{proof}The large deviation rate function in terms of $\chi$ and ${\xi_i}$:
\begin{eqnarray*}
 S(\tau,\xi)=\tau\ln(\chi)-C_{ab}x_a\xi_b
   =\tau\ln(\chi)-\ln\left(\frac{\chi}{\tau}\right)\left(
    (n+1)\xi_{n-1}-(n+1)\xi_n\right)-\\
   -\ln\left(
    \frac{\tau}{n+1}+\xi_{1}\right)(\xi_1+\xi_{n-1}-\xi_n)
    -\ln\left(
    \frac{\tau}{n+1}-\xi_{1}+\xi_{2}\right)(-\xi_{1}+\xi_{2}+\xi_{n-1}-\xi_{n})-\dots\\
   \dots
    -\ln\left(
    \frac{\tau}{n+1}-\xi_{n-2}+\xi_{n-1}\right)(-\xi_{n-2}+2\xi_{n-1}-\xi_{n})
    -\ln\left(
    \frac{\tau}{n+1}-\xi_{n}\right) .
    \end{eqnarray*}
To obtain expression for $\ln\chi$ in terms of ${\xi_i}$ we substitute solutions of (\ref{slnz}) into (\ref{chiz}):
\begin{eqnarray*}
    \chi=\frac{\chi}{\tau}\left(\frac{\tau}{n+1}+\xi_1\right)+\frac{\chi}{\tau}\left(\frac{\tau}{n+1}-\xi_1+\xi_2\right)+\dots+
    \frac{\chi}{\tau}\left(\frac{\tau}{n+1}-\xi_{n-2}+\xi_{n-1}\right)+\\
    \\+\frac{1}{\left(\frac{\chi}{\tau}\right)^n\left(
    \frac{\tau}{n+1}+\xi_1\right)
    \left(
    \dots+\frac{\tau}{n+1}-\xi_1+\xi_2\right)\dots
    \left(
    \frac{\tau}{n+1}-\xi_{n-2}+\xi_{n-1}\right)
     \left(
    \frac{\tau}{n+1}-\xi_{n}\right)
    }+
    \frac{\chi}{\tau}\left(\frac{\tau}{n+1}-\xi_n\right) ,
\end{eqnarray*}
which yields
\begin{equation}
    \chi^{n+1}=\frac{\tau^{n+1}}
    {
    \left(
    \frac{\tau}{n+1}+\xi_{1}\right)
    \left(
    \frac{\tau}{n+1}-\xi_{1}+\xi_{2}\right)
    \dots
   \left(
    \frac{\tau}{n+1}-\xi_{n-1}+\xi_{n}\right)
    \left(
    \frac{\tau}{n+1}-\xi_{n}\right)} ,
\end{equation}
and 
    \begin{eqnarray*}
    \label{lnc}
    \ln\chi=\frac{1}{n+1} (\ln\tau^{n+1}-\ln\left(\frac{\tau}{n+1}+\xi_{1}\right)
    -\ln\left(\frac{\tau}{n+1}-\xi_{1}+\xi_{2}\right)-\dots\\ 
    \dots-\ln\left(
    \frac{\tau}{n+1}-\xi_{n-1}+\xi_{n}\right)-
    \ln\left(
    \frac{\tau}{n+1}-\xi_{n}\right)) .
    \end{eqnarray*}
    Substituting $\ln\chi$ into the large deviation rate function we finally obtain:
    \begin{eqnarray*}
        S(\tau,\xi)=\tau\ln\tau-
        \left(\frac{\tau}{n+1}+\xi_1\right)\ln \left(\frac{\tau}{n+1}+\xi_1\right)-
        \left(\frac{\tau}{n+1}-\xi_1+\xi_2\right)\ln \left(\frac{\tau}{n+1}-\xi_1+\xi_2\right)-\dots\\
        \dots-
        \left(\frac{\tau}{n+1}-\xi_{n-1}+\xi_{n}\right)\ln \left(\frac{\tau}{n+1}-\xi_{n-1}+\xi_{n}\right)-
         \left(\frac{\tau}{n+1}-\xi_n\right)\ln \left(\frac{\tau}{n+1}-\xi_n\right)=\\
         =\tau\ln\tau-\sigma_1\ln\sigma_1-\sigma_2\ln\sigma_2-\dots-\sigma_{n+1}\ln\sigma_{n+1} ,
    \end{eqnarray*}
where
\[
\sigma_1=\frac{\tau}{n+1}+\xi_1,\;\; \dots\;\;,\;\; \sigma_i=\frac{\tau}{n+1}+\xi_i-\xi_{i-1},\;\; \dots\;\;,\;\; \sigma_{n+1}=\frac{\tau}{n+1}-\xi_n
\]
\end{proof}
\end{lemma}
The expression (\ref{slns1}) coincides with (\ref{slns}) which we obtained from hook length formula.

\subsection{The matrix $K$ and its determinant} Now let us compute the determinant of $K(x)$ when $x$ is the Legendre dual to $\xi$.

We have:
\[
    \frac{\partial S}{\partial \xi_i}=
    -\ln\left(\frac{\tau}{n+1}-\xi_{i-1}+\xi_i\right)
    +
    \ln\left(\frac{\tau}{n+1}-\xi_{i}+\xi_{i+1}\right) ,
\]
From here we derive 
\[
    \frac{\partial^2 S}{\partial^2 \xi_i}=
    -\frac{1}{\left(\frac{\tau}{n+1}-\xi_{i-1}+\xi_i\right)}
    -\frac{1}
    {\left(\frac{\tau}{n+1}-\xi_{i}+\xi_{i+1}\right)}= -\frac{1}{\sigma_i}-\frac{1}{\sigma_{i+1}},
\]
\[
    \frac{\partial^2 S}{\partial \xi_i\partial \xi_{i-1}}=
    \frac{1}{\left(\frac{\tau}{n+1}-\xi_{i-1}+\xi_i\right)}= \frac{1}{\sigma_i},
\]
\[
    \frac{\partial^2 S}{\partial \xi_i\partial \xi_{i+1}}=
    \frac{1}{\left(\frac{\tau}{n+1}-\xi_{i}+\xi_{i+1}\right)}=\frac{1}{\sigma_{i+1}} ,
\]
\[
    \frac{\partial^2 S}{\partial \xi_i\partial \xi_{j}}=0\;\;\;j\neq i,i-1,i+1 .
\]
From here we derive
\[
\det(K)=\frac{\tau}{\sigma_1\dots \sigma_{n+1}}
\]

\subsection{The factor $\Delta(x)e^{-(\rho, x)}$}
Finally, we need to compute $\Delta(x)$ and $e^{-(\rho,x)}$ .
The Weyl vector $\rho$ is the sum of fundamental weights $\omega_i$:
\[
    \omega_i=
    \frac{n+1-i}{n+1}e_1+ \dots +\frac{n+1-i}{n+1}e_i- \frac{i}{n+1}e_{i+1} \dots -\frac{i}{n+1}e_{n+1} .
\]
which gives
\[
    \rho=
    \frac{n}{2}e_1+\frac{n-2}{2}e_2+  \dots+\frac{-n+2}{2}e_{n-1}+\frac{-n}{2}e_{n+1} ,
\]
For $x=\sum_{i=1}^nx_i\alpha_i$, where $\alpha_i=e_i-e_{i+1}$, we have
\[
  e^{(\rho,x)}
  =e^{x_1}e^{x_2}\dots e^{x_n} .
\]

\begin{lemma}\label{1}
\begin{equation}
\label{ii}
    e^{(\rho,x)}=
    \sigma_1^{\frac{n}{2}}\sigma_{2}^{\frac{n-2}{2}}\sigma_{3}^{\frac{n-4}{2}}\dots  \;\;\;\sigma_{n+1}^{-\frac{n}{2}} .
\end{equation}
\end{lemma}

Indeed, we can use (\ref{slnz}) and (\ref{lnc}) to rewrite $e^{x_i}$ in terms of $\tau$ and ${\xi_i}$ and then in terms of $\sigma$ we obtain:
\begin{eqnarray*}
e^{x_1}=\sigma_1^{\frac{n}{n+1}} \sigma_2^{\frac{-1}{n+1}}\dots\dots\dots\dots\dots \;\;\;\sigma_n^{\frac{-1}{n+1}} \\
e^{x_2}=\sigma_1^{\frac{n-1}{n+1}} \sigma_2^{\frac{n-1}{n+1}}\sigma_3^{\frac{-2}{n+1}}\dots \dots\dots \;\;\;\sigma_{n+1}^{\frac{-2}{n+1}}\\
\dots\dots\dots\dots\dots\dots\dots\dots\\
e^{x_{i+1}}=\sigma_1^{\frac{n-i}{n+1}}\dots \sigma_{i+1}^{\frac{n-i}{n+1}}\sigma_{i+2}^{-\frac{i+1}{n+1}}\dots  \;\;\;\sigma_{n+1}^{-\frac{i+1}{n+1}}\\
e^{x_{n}}=\sigma_1^{\frac{1}{n+1}}\dots\dots\dots\dots\dots \sigma_{n}^{\frac{1}{n+1}}\sigma_{n+2}^{-\frac{n}{n+1}} ,
\end{eqnarray*} 
This proves lemma \ref{1}.

\begin{lemma}
\[
\Delta(x)=\frac{\prod_{i < j}\left(\sigma_i-\sigma_j\right)}{\prod_i^{n+1}\sigma_i^{\frac{n}{2}}}  
\]
\end{lemma}

\begin{proof} 
The Weyl denominator $\Delta(x)$ is the sum
\[
    \Delta(x)=\sum_{w}(-1)^{l(w)}e^{(w(\rho),x)}
\]
where $(-1)^{l(w)}$ is the signature of the permutation $w\in S_{n+1}$.  We have:
\[
    w(\rho)=\sum_{i=1}^{n+1}\frac{n+1-w(i)}{2}e_i,
\]
which implies
\[
         e^{(w(\rho),x)}=\prod_{i=1}^{n+1}\sigma_i^{\frac{n+1-w(i)}{2}}    
\]
    
To evaluate $\Delta$ we can bring the factor
 \[
        \prod_i^{n+1}\sigma_i^{\frac{n}{2}} .
 \]
 out of summation. The remaining  polynomial is the Vandermond determinant:
 \[
 \Delta(x)=\sum_{w}(-1)^{l(w)}e^{(w(\rho),x)}=\frac{\prod_{i < j}\left(\sigma_i-\sigma_j\right)}{\prod_i^{n+1}\sigma_i^{\frac{n}{2}}} \]
\end{proof}

\subsection{The comparison} Combining the results of the previous sections we 
conclude that 
\[
\epsilon^{\frac{n}{2}}\frac{\sqrt{\det K}}{(2\pi)^{\frac{n}{2}}}\Delta(x)e^{-(\rho,x)}
e^{\frac{1}{\epsilon}S(\tau,\xi)}(1+o(1))=
\left(\frac{\epsilon}{2\pi}\right)^{\frac{n}{2}}\tau^{\frac{1}{2}}
\prod_{i<j}\left(\sigma_i-\sigma_j\right)\prod_{i=1  }^{n+1}\sigma_i^{-n+i-3/2}
e^{\frac{1}{\epsilon} S(\tau,\sigma)}(1+o(1))
\]
where the function $S(\tau,\xi)$ on the left side is defined as the 
Legendre  transform of $\tau\ln(\chi_{\omega_1}(e^x))$ in $x$ and on the right
it is given by (\ref{slns}) and is derived from the hook formula.

\section{Intermediate scaling}
\label{Int}

In this section we consider the behavior of the character probability measure
\begin{equation}
  \label{stat1}
   p_\lambda^{(N)} =  \frac{m_{\lambda}(\{V_k\}, \{N_k\})\chi_\lambda(e^t)}{\prod_k \chi_{\nu_k}(e^t)^{N_k}} 
\end{equation}
near $t=0$ in the limit when $\epsilon\to 0$, $N_k=\frac{\tau_k}{\epsilon}$,$\lambda=\sqrt{\frac{x}{\epsilon}}b$, and $t=\sqrt{\frac{\epsilon}{x}}u$\footnote{There is an analogue of such asymptotic for every non regular value of $t$ (when $t\subset S\in \hf^{\star}\geq 0$ where $S$ is a stratum of the boundary of the cone $\hf^{\star}_{\geq 0}$). In this case transverse directions to $S$ should be scaled as $\sqrt{\epsilon}$.} where $b\in \hf^*_{\geq 0}$, $u\in  \hf_{\geq 0}$ 
are regular. 

\begin{theorem}
The pointwise asymptotic of $p_\lambda^{(N)}$ when $\epsilon\to 0$ is 
\[
\label{p4}
p_\lambda^{(N)}(t)=\left(\frac{\epsilon}{x}\right)^{\frac{r}{2}}\frac{\sqrt{\det 
B}}{(2\pi)^{\frac{r}{2}}}
\prod_{\alpha\in \Delta_+}
\frac{(b,\alpha)}
{(u,\alpha)}
\sum_w(-1)^we^{(b,w(u))}
e^{-\frac{(b,b)}{2}-\frac{(u,u)}{2}}(1+o(1)) . 
\]
\end{theorem}

\begin{proof}
The pointwise asymptotic of the multiplicity function is given by Theorem \ref{con-1}.
As it was shown in section \ref{planch}, the variables $x_a$ in the formula for the asymptotic  
for small $\xi_a$ are 
\[
x_b=\frac{\xi_b}{x}+O(\xi^2),
\]
where $x=\frac{1}{dim(\gf)}\sum_{k=1}^m \tau_k c_2(\nu_k)$ is as in section \ref{planch}.
In terms of the variable $b$ we have
\[
x_a=\sqrt{\frac{\epsilon}{x}}b_a+o(\sqrt{\epsilon}) .
\]

As a consequence for the Weyl denominator we have
\[
\Delta(x)=\prod_{\alpha\in \Delta_+} (e^{\frac{(x,\alpha)}{2}}-e^{-\frac{(x,\alpha)}{2}})=
\prod_{\alpha\in \Delta_+}
\left(
\sqrt{\frac{\epsilon}{x}}
\right)(a,\alpha)=
\left(\sqrt{\frac{\epsilon}{x}}
\right)^{|\Delta_+|}
\prod_{\alpha\in \Delta_+}
(a,\alpha) .
\] 
It is also easy to find the asymptotic of the character:
\begin{equation}
\chi_{\lambda}(e^t)=\frac{\sum_w(-1)^we^{(\lambda+\rho,w(t))}}{\Delta(t)}=
\frac{\sum_w(-1)^we^{(b,w(u))}}{\left(\sqrt{\frac{\epsilon}{x}}\right)^{|\Delta_+|}
\prod_{\alpha\in \Delta_+}
(u,\alpha)},   
\end{equation}
For small $\xi$ the large deviation rate function is
\[
S(\xi)=
-\frac{1}{2x}
\xi_bB_{bc}\xi_c
 +o(\xi^2),
\]
In terms of variable $b$ $\xi_a=\sqrt{\epsilon x} b_a$ and we have:
\[
S(\xi)=
-\frac{\epsilon}{2}
b_aB_{ac}b_c +o(\epsilon),
\]
For the matrix $K_{ab}$ as in section \ref{planch} we have:
\[
K_{ab}=\frac{B_{ab}}{x}+O(\epsilon),
\]
and therefore
\[
\det K=x^{-r}\det B(1+o(1).
\]

Finally, taking into account that $t=\sqrt{\frac{\epsilon}{x}}u$ we have
\[
\prod_k \chi_{V_{\nu_k}}(e^t)^{N_k}=\exp(\sum_k \frac{\tau_k}{\epsilon}\ln(dim(V_{\nu_k})+\frac{1}{2}(u,u)) .
\]
Here we used the identity
\[
tr_V(xy)=\frac{c_2(V)dim(V)}{dim{\gf}}(x,y) .
\]
where we assume that $V$ is irreducible, $c_2(V)$ is the value of the second Casimir
on $V$ and $(x,y)$ is the scalar product with respect to the Killing form. The identity follows from
the Schur lemma.

Now after the substitution of these asymptotical expressions into (\ref{stat1}) we obtain the desired asymptotic:
\begin{equation}\label{p2}
p_\lambda^{(N)}(t)=\left(\frac{\epsilon}{x}\right)^{\frac{r}{2}}\frac{1}{(2\pi)^{\frac{r}{2}}\sqrt{\det 
B}}
\frac{\prod_{\alpha\in \Delta_+}(b,\alpha)}{\prod_{\alpha\in \Delta_+}
(u,\alpha)}
\sum_{w\in W}(-1)^we^{(b,w(u))}e^{-\frac{(b,b)}{2}-\frac{(u,u)}{2}}
(1+o(1)) .
\end{equation}

\end{proof}

Let us check now that the function 
\[
p(b,u)=\frac{\sqrt{\det 
B}}{(2\pi)^{\frac{r}{2}}}
\frac{\prod_{\alpha\in \Delta_+}(b,\alpha)}{\prod_{\alpha\in \Delta_+}
(u,\alpha)}
\sum_{w\in W}(-1)^we^{(b,w(u))}e^{-\frac{(b,b)}{2}-\frac{(u,u)}{2}}
\]
is the density of a probability measure on $\hf^{\star}\geq 0$ ($b=\sum_ab_a\alpha_a,  b_a\geq 0$). It is clear that $p(b,u)\geq 0$,
so we only have to check the normalization. We naturally expect this to be true since $\sum_\lambda p^{(N)}_\lambda(t)=1$.
The following lemma checks it explicitly.

\begin{lemma} 
The function $p(b,u)$ is the density of a probability measure on $\hf^{\star}\geq 0$, i.e.
\[
    \int_{\hf\mathbb{R}^r\geq 0}p(b,u)db=1,
\]
where $db_1\dots db_r$ is the Euclidean measure on $\mathbb{R}^r$.
\end{lemma}

\begin{proof}
Let us calculate the normalization constant $N$ explicitly:
\begin{eqnarray*}
N=    \frac{\sqrt{\det 
B}}{(2\pi)^{\frac{r}{2}}}
\int_{\hf^{\star}\geq 0}
\prod_{\alpha\in \Delta_+}
\frac{(b,\alpha)}
{(u,\alpha)}
\sum_w(-1)^we^{(b,w(u))}
e^{-\frac{(b,b)}{2}-\frac{(u,u)}{2}}db\\
\frac{\sqrt{\det 
B}}{(2\pi)^{\frac{r}{2}}}
\frac{1}{\prod_{\alpha\in \Delta_+}
(u,\alpha)}
\frac{1}{|W|}
\int_{\hf^{\star}}
\prod_{\alpha\in \Delta_+}
(b,\alpha)
\sum_w(-1)^we^{(b,w(u))}
e^{-\frac{(b,b)}{2}-\frac{(u,u)}{2}}db .
\end{eqnarray*}

It is easy to compute the Gaussian integral
\begin{eqnarray*}
\frac{\sqrt{\det 
B}}{(2\pi)^{\frac{r}{2}}}\int_{\hf^{\star}}
e^{-\frac{(b,b)}{2}+(b,u)}db=\\
=\frac{\sqrt{\det 
B}}{(2\pi)^{\frac{r}{2}}}\int_{\hf^{\star}}
e^{-\frac{(b-u,b-u)}{2}+\frac{(u,u)}{2}}db= 
e^{\frac{(u,u)}{2}} .
\end{eqnarray*}
We have
\[
   \prod_{\alpha\in \Delta_+}
(b,\alpha)
e^{(b,u)} =
\prod_{\alpha\in \Delta_+}
(b,\alpha)
e^{b_aB_{ab}u_b}=
\prod_{\alpha\in \Delta_+}
(b,\alpha)
e^{b_a u^a}=
\prod_{\alpha\in \Delta_+}
\left(\frac{\partial}{\partial u},\alpha\right)e^{(b,u)}=
D(u)e^{(b,u)},
\]
where $D_u=\prod_{\alpha\in \Delta_+}
\left(\frac{\partial}{\partial u},\alpha\right)$.

The normalization constant $N$ now can be written as
\[
N=\frac{1}{|W|\prod_{\alpha\in \Delta_+}
(u,\alpha)}\sum_{w}(-1)^wJ(w(u)) ,
\]
where
\[
J(u)=e^{-\frac{(u,u)}{2}}D_ue^{\frac{(u,u)}{2}} .
\]
The following identity is easy to check:
\[
\sum_{w}(-1)^wJ(w(u))=|W|\prod_{\alpha\in \Delta_+}
(u,\alpha) .
\]

Thus
\[
N=1 .
\]
\end{proof}

Note that (\ref{p2}) implies a weak convergence of the sequence of character measures $\{p^{(N)}_\lambda(t)\}$ to $p(b,u)db$.

\section{Nonlinear PDE for the rate function $S(\tau, \xi)$}
\label{PDE}
Let $V=\oplus_{\mu}V(\mu)$ be the weight decomposition of $V$ and $d_{\mu}=\dim(V_{\mu})$. For the weight $\lambda$, which is well inside the main Weyl chamber, we have the following decomposition of the tensor product
\begin{equation}
\label{tenn}
    V\otimes V_{\lambda}\simeq
    \oplus_{\mu\in wt(V)}V_{\lambda+\mu}^{\oplus d_{\mu}} .
\end{equation}
were $wt(V)$ is the set of weights which occur in the representation $V$.

Let $ m_\lambda^{(N)}$ be the multiplicities in
\[
    V^{\otimes N}\simeq 
    \oplus_{\lambda}
    V_{\lambda}
    ^{\oplus m_\lambda^N} .
\]
the decomposition (\ref{tenn}) gives the difference equation for multiplicities
\[
    m_\lambda^{(N+1)}=\sum_{\mu\in wt(V)}d_{\mu}m_{\lambda-\mu}^{(N)} ,
\]
where we assume again that $\lambda$ is well inside the main Weyl chamber.

Now, passing to the limit $\epsilon\to 0$ in this difference equation with $N=\frac{\tau}{\epsilon}$, $\lambda=\frac{\xi}{\epsilon}$ we derive the following nonlinear PDE for the large deviation rate function $S$:
\begin{equation}
\label{part}
 e^{\partial_\tau S}=\sum_{\mu\in wt(V)}d_\mu e^{-\sum_a\mu_a\frac{\partial S}{\partial \xi_a}} .   
\end{equation}
Here we used the basis of simple roots in $\hf^*$: $\mu=\sum_a\mu_a\alpha_a, \;\;\xi=\sum_a\xi_a\alpha_a$.

The following proposition is a direct verification of the equation (\ref{part}) for $S(\tau,\xi)$ determined by (\ref{s}) .
\begin{proposition}
The function $S(\tau,\xi)$ defined in (\ref{Sf}) satisfies the differential equation (\ref{part}).
\begin{proof}
By definition
\[
    S(\tau,\xi)=\tau\ln\chi_{\nu}(e^x)-(x,\xi) ,
\]
where $x$ is the unique solution to the equation:
\[
\tau\frac{\partial}{\partial x_a}\chi_{\nu}(e^x)=\chi_{\nu}(e^x)\sum_b B_{ab}\xi_b.
\]
Differentiating this with respect to $\tau$ we have
\[
\frac{\partial}{\partial x_a}\ln\chi_{\nu}(e^x)+\frac{\partial x_b}{\partial \tau}\frac{\partial^2}{\partial x_a \partial x_b}f(\tau,x)=0.
\]
i.e.
\[
    D_{ab}\frac{\partial x_b}{\partial \tau}=-\frac{1}{\tau}B_{ab}
    \xi_b .
\]
From here we obtain
\[
    \partial_{\tau}S=\ln\chi_{\nu}(e^x)+\tau\frac{\partial x_b}{\partial \tau}\frac{\partial }{\partial x_b}\ln\chi_{\nu}(e^x)-
    \frac{\partial x_a}{\partial \tau}B_{ab}\xi_b=\ln\chi_{\nu}(e^x) .
\]
On the other hand
\[
    \frac{\partial S}{\partial\xi_a}=
      \frac{\partial x_b}{\partial\xi_a}\tau
      \frac{\partial }{\partial x_b}\ln\chi_{\nu}(e^x)-
      \frac{\partial x_b}{\partial\xi_a}B_{bc}\xi_c-x_bB_{ba} = -x_bB_{ba} .
\]
i.e.
\[
    e^{-\mu_a\frac{\partial S}{\partial \xi_a}}=e^{(\mu,x)}
\]
Thus, both sides of (\ref{part}) are equal to $\chi_\nu(e^x)$ and we verified (\ref{part}).
\end{proof}
\end{proposition}

\section{Conclusion}
\label{concl}

This paper demonstrats that in large tensor products
the statistic of irreducible components with respect to the 
character distribution almost does not depend
on which representations being multiplied. But it depends significantly on
whether the parameter $t$ in the character distribution is generic or special. 

Note that one can associate natural Markov process with 
the decomposition of tensor powers $V^{\otimes N}$ of a 
finite dimensional $\gf$-module $V$. The transition probabilities in 
such a process are 
\[
M_{\lambda, \mu}=\frac{\chi_\lambda(e^t)b_{\lambda,\mu}}{\chi_{\mu}(e^t)\chi(e^t)}
\]
where $\chi(e^t)$ is the character of the representation $V$ evaluated at $e^t$
and $b_{\lambda,\mu}$ are multiplicities of irreducible components in the tensor product
\[
V\otimes V_{\lambda}\simeq  \oplus_\mu V_\mu^{\oplus_\mu b_{\lambda, \mu}}
\]
Note that if $\lambda$ is sufficiently inside the positive Weyl chamber, $b_{\lambda, \lambda-\mu}=d_\mu$
where $d_\mu$ is the multiplicity of weight $\mu$ in $V$.

The character probability distribution (\ref{stat}) is a result of the Markov evolution of 
character measure of the trivial representation:
\[
p_\lambda^{(0)}(t)= \delta_{\lambda, 0}
\]
i.e.
\[
p^{(N)}=M^Np^{(0)}
\]

This paper can be regarded as a study of this Markov process\footnote{Or of its slight generalization when we consider $\otimes_k V_{\nu_k}^{\otimes N_k}$ instead of $V^{\otimes N}$.} . This process 
also can be regarded as a random walk on a lattice domain. For results in
this direction see \cite{B1}; also \cite{M} and references therein.
From this point of view the results on the week convergence of character 
measures can be interpreted as follows:

\begin{itemize}
\item For regular $t$, as $N\to \infty$, the expectation value of $\lambda$ behave as 
$\mathbb{E}(\lambda)=N\nabla\ln(\chi(e^t))(1+o(1)$ where $\nabla$ is the gradient (with respect to
the Killing form on the weight space). As $N\to \infty$ the distribution converges to the Gaussian 
distribution around the expectation value with the dispersion behaving as $\sqrt{N}$.
\item For $t=0$ and $N\to \infty$ the expectation value of $\lambda$ vanishes and
the asymptotical distribution is given by (\ref{Pl}) rescaled such that the dispersion is
proportional to $\sqrt{N}$.
\end{itemize}

Similar interpretation can be given for the intermediate scaling.

Below we will outline some of the future directions and problems
in the asymptotic representation theory that are naturally related 
to this paper.

\begin{itemize}

\item  We studied the asymptotic of dimensions of irreducible components and
the corresponding probability distributions when $t$ is either regular or maximally irregular, i.e. $t=0$.
When $t$ is not regular the asymptotical measure is given by a different formula involving the centralizer
of the action of $W$ on $\xi$. These results will be presented in a separate publication.

\item Truncated  tensor products, also known as fusion products appear in representation  theory of quantum groups 
at roots of unity and plays an important role in constructing modular categories. The latter are 
instrumental in conformal filed theory and in topological quantum  filed theory. 
The study of statistics of irreducible components for truncated tensor products
is an important problem by itself and an important  step in understanding the
semiclassical limit of corresponding topological quantum field theory. See \cite{D} and
references therein for existing results on fusion products and corresponding random walks.

\end{itemize}

Other interesting  problems expanding the results announced here are related to multiplicities 
of irreducible components in tensor products of Lie superalgebras and of finite groups  of Lie type,
i.e. groups similar to $SL_n(F_p)$. In these cases, as well as in the case of quantum groups at roots of unity,
tensor products of irreducibles have both irreducible components and blocks of irreducibles.
The study of statistics of blocks in large tensor products is another interesting problem that is largely open
for investigation.


\begin{thebibliography}{99}


\bibitem{BR1} A. Berele, A. Regev, {\it Asymptotic of Young tableaux in the $(k,l)$ hook}, 
Contemporary Mathematics, {\bf 537}, pp. 71-83.

\bibitem{BR2} A. Berele, A. Regev, {\it Hook Young Diagrams with Applications to Combinatorics and to 
Representations of Lie algebras}, Advances in Mathematics, {\bf 64}, pp. 118-175.

\bibitem{B1} Ph. Biane, {\it Miniscual weights and random walks on lattices}, Quant. Prob. Rel. Topics, v. 7 (1992), 51-65.

\bibitem{B} Ph. Biane, {\it Estimation asymptotique des multiplicités  dans les puissances tensorielles d'un $\gf$-module,}
 C.R. Acad. Sci. Paris, t. 316, Serie I, p. 849-852, 1993.
 
\bibitem{BOO} A. Borodin, A. Okounkov and G. Olshanski, {\it Asymptotics of Plancherel measures for symmetric groups},
Journal of the American Mathematical Society, {\bf 13}, 3, pp. 481--515, 2000.

\bibitem{BG} A. Bufetov, V. Gorin, {\it Fourier transform on high-dimensional unitary groups with applications to random tilings}, arXiv:1712.09925

\bibitem{D} М. Defossеux, {\it Fusion coefficients and random walks in alcoves}, arXiv:1307.3830.

\bibitem{Du} N. G. Duffield, {A large deviation principle for the reduction of product representations}, Proceedings of the American Mathematical Society, v. 109, n. 2, 1990.

\bibitem{K1} S. Kerov, {\it On asymptotic distribution of symmetry types of high rank tensors},
Zapiski Nauchnykh Seminarov POMI, {\bf 155}, pp. 181-186, 1986.

\bibitem{K2} S. Kerov, Asymptotic Representation Theory of the Symmetric Group and its Applications in Analysis,
Translations of Mathematical Monographs, vol. 219, 2003.

\bibitem{LW} M. Keyl, R.F. Werner, {\it Estimating the spectrum of a density operator}, Phys. Rev. A 64, (2001) 052311.

\bibitem{LS} B. Logan, and L. Shepp,  {\it A variational problem for random Young tableaux}, Advances in mathematics,
{\bf 26}, 2, pp. 206--222, 1977.

\bibitem{M} P.-L. Méliot, {\it A central limit theorem for the characters of the infinite symmetric group and of the infinite Hecke algebra},   arXiv:1105.0091.

\bibitem{NP} A. Nazarov,  Postnova, The limit shape of a probability measure on a tensor product of modules 
on the $B_n$ algebra, {\it Zapiski Nauchnykh Seminarov POMI}, {\bf 468}, 82--97, 2018.

\bibitem{ST} M. Stolz, T. Tate  {\it Asymptotics of matrix integrals and tensor invariants of compact Lie groups}, Proc. Amer. Math. Soc. 136 (2008), 2235-2244. arXiv:math/0610720.

\bibitem{TZ} T.Tate, S. Zelditch, {\it Lattice path combinatorics and asymptotics of multiplicities of weights in tensor powers},  J. Funct. Anal. 217 (2004), no. 2, 402--447.  arXiv:math/0305251.

\bibitem{VK1} A. Vershik and S. Kerov, {\it Asymptotics of Plancherel measure of 
symmetrical group and limit form of young tables},
 Doklady Akademii Nauk SSSR, {\bf 233}, 6, pp 1024--1027, 1977.
 
\bibitem{VK2} A. Vershik and S. Kerov, {\it Asymptotic of the largest and the typical dimensions 
of irreducible representations of a symmetric group}, Functional analysis and its applications, 
{\bf 19}, 1, pp. 21--31, 1985

\bibitem{V} A. Vershik (Ed.), Asymptotic Combinatorics with Applications to Mathematical Physics,
A European Mathematical Summer School held at the Euler Institute, St. Petersburg, Russia, July 9-20, 2001,
Lecture Notes in Mathematics, 1815.

\end{thebibliography}
\end{document}